\documentclass[11pt]{article}
\usepackage{amsmath}
\usepackage{amssymb,amsfonts,amsmath,amsthm}
\usepackage{epsfig}
\usepackage{amssymb}
\usepackage{changepage}
\usepackage{latexsym}
\usepackage{bbm}
\usepackage{multirow}
\usepackage{latexsym}
\usepackage{graphicx}
\usepackage{relsize}
\usepackage{leftidx}
\newmuskip\pFqmuskip
\newtheorem{theorem}{Theorem}[section]
\newtheorem{thm}[theorem]{Theorem}

\newtheorem{lem}[theorem]{Lemma}

\newtheorem{cor}[theorem]{Corollary}

\makeatletter \@addtoreset{equation}{section}

\newcommand{\qbinom}[2]{\genfrac{[}{]}{0pt}{}{#1}{#2}}

\newcommand{\Qrki}{Q_1(b \Mid \br;\bk)}
\newcommand{\Qrkii}{Q_2(b \Mid \br;\bk)}
\newcommand{\Qrka}{Q_3(b \Mid \br;\bk)}
\newcommand{\Qrkb}{Q_4(b \Mid \br;\bk)}
\newcommand{\Mid}{\:|\:}  
\newcommand{\br}{\mathbf r}
\newcommand{\bk}{\mathbf k}
\newcommand{\Erk}{E_{\br,\bk}}
\newcommand{\lrq}[3]{\left(\frac{#1}{#2}#3\right)}
\newcommand*\pFq[6][8]{%
  \begingroup 
  \pFqmuskip=#1mu\relax
  \mathcode`\,=\string"8000
  \begingroup\lccode`\~=`\,
  \lowercase{\endgroup\let~}\pFqcomma
  {}_{#2}F_{#3}{\left[\genfrac..{0pt}{}{#4}{#5};#6\right]}%
  \endgroup
}
\newcommand{\pFqcomma}{\mskip\pFqmuskip}

\def\CT{\mathop{\mathrm{CT}}}
\def\LC{\mathop{\mathrm{LC}}}

\begin{document}
	
	\def\CC{\mathbb{C}}
	\title{Two Families  of Constant Term Identities}
    \author{Keru Zhou \thanks{~School of Mathematics and Statistics, HNP-LAMA, Central South University,
    		Changsha, Hunan 410083, P. R. China. E-mail address: {\it krzhou1999@knights.ucf.edu}.}}

    \date{}
	\maketitle

	\begin{abstract}
     In 1985, Bressoud and Goulden derived the formula for the constant term in  $\prod_{(i,j)\in T} \frac{x_j}{x_i}\\\prod_{0\le i<j \le n}(\frac{x_i}{x_j})_{a_i}(\frac{qx_j}{x_i})_{a_j-1}$,  where $T \subseteq \{(i,j)\mid 0\le i<j \le n\}$.
     This result implies the Andrews' $q$-Dyson identity. In 2006, Gessel and Xin proved the $q$-Dyson identity by considering both sides of the equality as polynomials  in $q^{a_0}$. We use this approach to determine the coefficients of $x_0/x_1$ and $x_0/x_2$ in Laurent polynomials studied by Bressoud and Goulden.
	\end{abstract}
	\section{Introduction}
	
	Throughout this paper, we adopt the following notation :
	\begin{multline*}
	\shoveright{\text{$n$ is a positive integer, $a_0$ is an integer and $a_1,a_2,\cdots,a_n$ are nonnegative integers;}\\}
	\shoveright{\mathbf{a}:=(a_0,a_1,\ldots,a_n);}
	\shoveright{\mathbf{x}:=(x_0,x_1,\ldots,x_n);}
   \shoveright{(z)_{\infty}:=(1-z)(1-zq)\cdots;\\}
	\shoveright{
    \text{for $k$ an integer \quad}\\
   (z)_k:=\frac{(z;q)_{\infty}}
{(zq^k;q)_{\infty}}=
\begin{cases}
(1-z)(1-zq)\cdots (1-zq^{k-1}) & \text{if $k\geq 0,$}
\\[3mm]
\displaystyle
\frac{1}{(1-zq^k)(1-zq^{k+1})\cdots(1-zq^{-1})}
& \text{if $k<0$;}
\end{cases}\\}
	\shoveright{\genfrac{[}{]}{0pt}{}{m}{n}=\frac{(q^{m-n+1})_n}{(q)_n},
     \quad \text{for $m$ an integer, which follow from the definition of the $q$-binomial coefficient;}\\}
	\shoveright{D_n(\mathbf{x},\mathbf{a},q):=\prod_{0\leq i<j\leq n}
			\left(\frac{x_i}{x_j}\right)_{\!\!a_i}
			\left(\frac{x_j}{x_i}q\right)_{\!\!a_{j-1}};\quad\quad\quad\quad \quad \quad \quad \quad \quad \quad\quad\quad(q\mbox{-Dyson style product})\\}
	\shoveright{\CT_{\mathbf{x}}F(\mathbf{x})\ \mbox{the
	constant term of the series $F(\mathbf{x})$;}}
	\end{multline*}
	
	$
	   E:={\{(i,j)\mid1\leq i<j\leq n\}};
	$

	In 1891, Dixon \cite{dixon1891} provide the following identity. 
	
	\begin{thm}[Dixon's identity]
	Let n be a positive integer. Then
	\begin{equation}
	    \sum_{k=-n}^{n}(-1)^k {{2n}\choose{k+n}}^3 = \frac{(3n)!}{{(n)!}^3}.
	\end{equation}
	
	\end{thm}
	
	In 1962, Freeman Dyson \cite{dyso1962} conjectured the following
	identity.
	\begin{theorem}\label{t-dyson}
		For nonnegative integers $a_0,a_1,\ldots ,a_n$,
		\begin{equation*}
		\CT_{\mathbf{x}} \prod_{0\le i\ne j \le n}
		\left(1-\frac{x_i}{x_j}\right)^{\!\!a_i} =
		\frac{(a_0+a_1+\cdots+a_n)!}{a_0!\, a_1!\, \cdots a_n!}.
		\end{equation*}
	\end{theorem}
	
	Dyson's conjecture was first proved independently  by Wilson \cite{kenn1962} and Gunson \cite{gunson}. An elegant
	recursive proof was given by Good \cite{good1970}.
	
	George Andrews \cite{geor1975} conjectured a $q$-analog of the
	Dyson conjecture in 1975.
	\begin{thm}\label{thm-dyson}
		For nonnegative integers $a_0,a_1,\dots,a_n$,
		\begin{align*}
		\CT_{\mathbf{x}}\prod_{0\leq i<j\leq n}\Big(\frac{x_i}{x_j}\Big)_{a_i}\Big(\frac{qx_j}{x_i}\Big)_{a_j}=\frac{(q)_{a_0+\cdots+a_n}}{(q)_{a_0}(q)_{a_1}\cdots(q)_{a_n}}.
		\end{align*}
	\end{thm}
	
	Andrews' $q$-Dyson conjecture attracted much interest
	\cite{aske1980, kade1985, stan1983, stan1984,
		stem1988}. It was first proved, combinatorially, by
	Zeilberger and Bressoud \cite{bress1985} in 1985. In 2004,
	Gessel and Xin \cite{gessxin2006} gave a very different proof 
	using formal Laurent series. For related constant term identity,
	one can refer to
	\cite{goul1985, kade1998,  sill20061, sill20062,
		stem1987}.
	
A tournament $T$ on $n$ vertices (or on the set $\{1 \cdots,n\}$) is a set of ordered pairs $(i,j)$ such that $1 \le i\neq j \le n$ and $(i, j)\in T$ if and only if $(j, i)\notin T$.
Equivalently, $T$ can be thought of as a directed graph with vertices $1,\dots,n$ and edges directed from $i$ to $j$ for all $(i, j)\in T$. Thus we write $i\rightarrow j$ if $(i, j)\in T$.
The tournament $T$ is transitive if the relation $\rightarrow$ is transitive. Equivalently, $T$ is transitive if it contains no cycles $(i\rightarrow j \rightarrow k \rightarrow i)$. Otherwise, $T$ is nontransitive.
Let $S$ be a tournament on a subset of $\{1,\cdots,n\}$. Define $T\overline{S}$ to be the tournament on $\{1,\cdots,n\}$ such that if $(i,j)\notin S$, then $(i,j)\in T\overline{S}$, otherwise $(j,i)\in T\overline{S}$. 

 By extending the method used by Zeilberger-Bressoud  \cite[Theorem 2.9]{bress1985}, Bressoud and Goulden generalized the $q$-Dyson constant identity \cite{goul1985}.
	\begin{thm}\label{t-main-thm}
		Let $T$ be a subset of $\{(i,j)\mid 1\le i<j\le n\}$.
		\begin{multline}\label{e-main}
		  	\CT_{\mathbf{x}}\prod_{(i,j)\in T}\frac{x_j}{x_i}\prod_{1\leq i<j\leq n}\Big(\frac{x_i}{x_j}\Big)_{a_i}\Big(\frac{qx_j}{x_i}\Big)_{a_j-1}\\=\begin{cases}
			0& \text{if $E\overline{T}$ is nontransitive,}\\ (-1)^{|T|}\frac{(q)_{a_1\cdots+a_n}}{(q)_{a_1}(q)_{a_2}\cdots(q)_{a_n}}\prod_{i=1}^{n}\frac{1-q^{a_i}}{1-q^{\Psi_i}}
 & \text{if $E\overline{T}$ is transitive},
			\end{cases}  
		\end{multline}

where $|T|$ stands for the number of elements of $T$ and $\Psi_k=\sum_{i=1}^k a_{\sigma(i)}$.
	\end{thm}
In 2015, using multivariable Lagrange
interpolation  K\'{a}rolyi,  Lascoux and  Warnaar 
\cite[Theorem 1.2]{karo2015} gave an identity equivalent to Theorem~\ref{t-main-thm} . 

Let $T$ be transitive tournament on $\{1,\cdots,n\}$ and let $\sigma$ be a permutation of $\{1,\cdots,n\}$. We say that $\sigma$ is the winner permutation for $T$ if $\sigma(1)\rightarrow \sigma(2)\rightarrow \cdots \rightarrow \sigma(n)$ in $T$.
 
%
In this paper, we obtain a closed form formula for two coefficients of $D_n(\mathbf{x},\mathbf{a},q)$.
\begin{thm}[Main theorem 1]\label{mainthm2}
Let $n$ be an integer such that $n\geq 2$ and $T_1$ be a subset of $\{(i,j)\mid2\leq i<j \leq n\}$.
Then
	\begin{multline}
	 \CT_{\mathbf{x}}\frac{x_0}{x_1}\prod_{(i,j)\in T_1}\frac{x_j}{x_i} D_n(\mathbf{x},\mathbf{a},q)\\
=\begin{cases}
	 0&\text{$E\overline{T_1}$ is nontransitive,}\\ (-1)^{|T_1|}\Big(\frac{q^{a_{\sigma(1)}}-q^{a_{\sigma(2)}}}{1-q^{a_0+a_{\sigma(2)}}}\Big)\frac{(q)_{a_0+a_1\cdots+a_n}}{(q)_{a_0}(q)_{a_1}\cdots(q)_{a_n}}
	 \prod_{i=1}^{n}\frac{1-q^{a_i}}{1-q^{a_0+{\Psi_i}}} & \text{ $E\overline{T_1}$ is transitive,}
	 \end{cases}
	\end{multline}
	$\sigma$ is the winner permutation for $E\overline{T_1}$ and $\Psi_k=\sum_{i=1}^k a_{\sigma(i)}$.
	
	\end{thm}
	\begin{thm}[Main theorem 2]\label{mainthm3}
	Let n be an integer such that $n\geq 3$ and $T_2$ is a subset of $\{(i,j)\mid 3\leq i<j\leq n) \}$. Then
	\begin{multline}
	 \CT_{\mathbf{x}}\frac{x_0}{x_2}\prod_{(i,j)\in T_2}\frac{x_j}{x_i} D_n(\mathbf{x},\mathbf{a},q)\\
	 =\begin{cases}
	 0 &\text{$E\overline{T_2}$ is nontransitive,}\\
	 (-1)^{|T_2|}\Big(\frac{(1+q^{a_{\sigma(1)}})(q^{a_{\sigma(2)}}-q^{a_{\sigma(3)}})}{1-q^{a_0+a_{\sigma(1)}+a_{\sigma(3)}}}\Big)
	 \frac{(q)_{a_0+a_1\cdots+a_n}}{(q)_{a_0}(q)_{a_1}\cdots(q)_{a_n}}\prod_{i=1}^{n}\frac{1-q^{a_i}}{1-q^{a_0+\Psi_i}} & \text{ $E\overline{T_2}$ is transitive,}
	 \end{cases}
	\end{multline}
	where 
	$\sigma$ is the winner permutation for $E\overline{T_2}$ and $\Psi_k=\sum_{i=1}^k a_{\sigma(i)}$.
	 
	\end{thm}
	
    \begin{cor} \label{main3}
   Let $\sigma$ be a permutation  of  $(1,2,\cdots,n)$.\\
	$\bf{(i)}$ For $n$ be a positive integer such that $n\geq 3$, 
	\begin{multline}
	    \CT_{\mathbf{x}}\frac{x_{\sigma(1)}}{x_{\sigma(2)}}\prod_{(i,j)\in T}\frac{x_i}{x_j}
	\prod_{1\leq i<j\leq n}\Big(\frac{x_i}{x_j}\Big)_{a_i}\Big(\frac{qx_j}{x_i}\Big)_{a_j-1}\\
	= (-1)^{|T|}\Big(\frac{q^{a_{\sigma(2)}}-q^{a_{\sigma(3)}}}{1-q^{a_{\sigma(1)}+a_{\sigma(3)}}}\Big)\frac{(q)_{a_1+a_2+\cdots+a_n}}{(q)_{a_1}(q)_{a_2}\cdots (q)_{a_n}} \prod_{i=1}^n\frac{1-q^{a_{\sigma(i)}}}{1-q^{\Psi_i}}. \\ 
	\end{multline}

	$\bf{(ii)}$ For $n$ be a positive integer such that $n\geq 4$, 
    \begin{multline}
         \CT_{\mathbf{x}}\frac{x_{\sigma(1)}}{x_{\sigma(3)}}\prod_{(i,j)\in T}\frac{x_i}{x_j}
	\prod_{1\leq i<j\leq n}\Big(\frac{x_i}{x_j}\Big)_{a_i}\Big(\frac{qx_j}{x_i}\Big)_{a_j-1}\\
	= (-1)^{|T|}
	\Big(\frac{(1+q^{a_{\sigma(2)}})(q^{a_{\sigma(3)}}-q^{a_{\sigma(4)}})}{1-q^{a_{\sigma(1)}+a_{\sigma(2)}+a_{\sigma(4)}}}\Big)\frac{(q)_{a_1+a_2+\cdots+a_n}}{(q)_{a_1}(q)_{a_2}\cdots (q)_{a_n}} \prod_{i=1}^n\frac{1-q^{a_{\sigma(i)}}}{1-q^{\Psi_i}},
    \end{multline}
 
	where
	$T:=\{(\sigma(i),\sigma(j))\mid i<j\quad and \quad \sigma(i)>\sigma(j)\}\quad
	and\quad \Psi_k :=\sum_{i=1}^k a_{\sigma(i)}.$
	\end{cor}
	
	The proofs of Theorems~\ref{mainthm2} and \ref{mainthm3} are similar to the Gessel-Xin proof of the $q$-Dyson term constant identity \cite{gessxin2006}. The basic idea is the well-known fact that to establish the equality of
    two polynomials of degree at most $d$, it is sufficient to prove that they are equal at
    $d + 1$ points. As is the case generally, points at which the polynomials vanish
     are most easily dealt with.
	
	For fixed $a_1,a_2,\cdots,a_n$, the constant term can be regarded as a polynomial of degree $d$ in variables $q^{a_0}$. Then we use the Gessel-Xin technique to check that both sides have zeroes at the same $d$ values of $q^{a_0}$
	and that the case $q^{a_0}=1$ reduces 
	to the identity with $n$ replaced by $n-1$. As in the Gessel-Xin proof, the difficulty lies in verifying that the constant term
	is zero at the desired values. The fact that the product is multiplied $x_0/x_1$ or
	$x_0/x_2$ forces several modifications to the argument employed by Gessel and
	Xin, including the need to put an upper bound on the degree of the Laurent
	polynomial rather than being able to compute its exact degree.

	This paper is organized as follows. In section 2, we estimate the constant term degree in $q^{a_0}$ on the left-side and calculate the degree of $q^{a_0}$ and its zero points on right-side for Theorem \ref{mainthm2} and Theorem \ref{mainthm3}. In section 3, we introduce some basic notions and lemmas of \cite{gessxin2006} in a
    generalized form.
    At the same time, we describe the field of iterated Laurent series and the use
    of partial fraction decompositions as main tools for estimating constant terms.
    In section 4, we introduce new techniques to prove the nontransitive cases of Theorems \ref{mainthm2} and \ref{mainthm3}. In section 5, we first prove these theorems when $T_1$ and
    $T_2$ are empty, then inductively build up the general case. Finally, we prove and
    apply Corollary\ref{main3}.

	\section{Basic Facts}
	Dyson's conjecture, Andrews' $q$-Dyson conjecture, and their
	relatives are all constant terms of certain Laurent polynomials.
	In fact, larger rings and fields will be encountered when evaluating
	them. We use the follow  notation from \cite{gessxin2006}. In
	order to prove our Main Theorem 1 and 2, we make some generalizations
	that need detailed explanation.

	\begin{lem}\label{pd}
		Let $a_1,\ldots,a_n$ be  nonnegative integers
		and $k\leq |a|$, $k\in \mathbb{Z}$ and let $L(x_1,\ldots,x_n)$ be a Laurent polynomial with $L(x_1,\ldots,x_n)$ independent of
		${a_0}$ and $x_0$. Then   the constant term
		\begin{align}\label{p1}
		\CT_\mathbf{x} x_0^k L(x_1,\dots,x_n) D_n(\mathbf{x},\mathbf{a},q)
		\end{align}
		is a polynomial of degree at most $|a|-k-n$ in $q^{a_0}$, where $|a|=\sum_{i=1}^na_i$.
	\end{lem}
	\begin{proof}
		The basic idea is from \cite[Lemma 2.2]{zhou2009}.\\
		First of all, it is straightforward to show that
		\begin{align*}
		\left(\frac{x_0}{x_j} \right)_{\!\!a_0}\!\left(\frac{x_j}{x_0}q
		\right)_{\!\!a_j-1}&=q^{\binom{a_j}{2}}\left(-\frac{x_j}{x_0}\right)^{\!\!{a_j-1}}
		\!\left(\frac{x_0}{x_j} q^{-a_j+1}\right)_{\!\!a_0+a_j-1},
		\end{align*}
		for all integers $a_0$. We can regard the two sides as Laurent
		series in $x_0$. Rewrite \eqref{p1} as
		\begin{align}
		\label{e-product} \CT_{\mathbf{x}}\: x_0^kL_1(x_1,\dots
		,x_n) \prod_{j=1}^n q^{\binom{a_j}{2}}
		\left(-\frac{x_j}{x_0}\right)^{\!\!a_j-1}\! \left(\frac{x_0}{x_j}
		q^{-a_j+1}\right)_{\!\!a_0+a_j-1}\! ,
		\end{align}
		where $L_1(x_1,\dots,x_n)$ is a Laurent polynomial in
		$x_1,\dots, x_n$ and $L_1(x_1,\dots,x_n)$ is independent of $x_0$ and $a_0$.

		Using the famous $q$-binomial theorem \cite[Theorem
		2.1]{andr1976}, we obtain an identity
		\begin{align}
		\label{e-qbinomial} \frac{(bz)_\infty}{(z)_\infty} =
		\sum_{k=0}^\infty \frac{(b)_k}{(q)_k} z^k.
		\end{align}
		Using $q^{-n}$ instead of $b$ and replacing  $z$ by $uq^n$ in \eqref{e-qbinomial}, we get
		\begin{align}\label{e-qbinomialn}
		(u)_n=\frac{(u)_\infty}{(uq^n)_\infty}=
		\sum_{k=0}^\infty q^{k(k-1)/2}\qbinom{n}{k} (-u)^k
		\end{align}
		for all integers $n$, where $\qbinom{n}{k}=\frac{(q)_n}{(q)_k(q)_{n-k}}$ is the $q$-binomial coefficient.

		Applying \eqref{e-qbinomialn}, we have
		that for $1\le j\le n$, {\small
			\begin{align*} q^{\binom{a_j}{2}}
			\left(-\frac{x_j}{x_0}\right)^{\!\!a_j-1} \left(\frac{x_0}{x_j}
			q^{-a_j+1}\right)_{\!\!a_0+a_j-1}
			=\sum_{k_j\geq 0}C(k_j) \qbinom{a_0+a_j-1}{k_j}
			x_0^{k_j-a_j+1}x_j^{a_j-k_j-1},
			\end{align*}}
		where $C(k_j)=(-1)^{k_j+a_j-1}q^{\binom{a_j}2 + \binom {k_j}2
			-k_j(a_j-1)}$.
		
		Expanding the product in \eqref{e-product} and  extracting the constant
		term in $x_0$, we find that \eqref{p1} becomes
		{\small\begin{align} \label{e-midle} \sum_{\mathbf{k}}
			\qbinom{a_0+a_1-1}{k_1}\qbinom{a_0+a_2-1}{k_2}\cdots
			\qbinom{a_0+a_n-1}{k_n} \CT_{x_1,\dots,x_n} L_2(x_1,\dots
			,x_n;\mathbf{k}),
			\end{align}}
		where $L_2(x_1, \dots, x_n;\mathbf{k})$ is a Laurent
		polynomial in $x_1, \dots, x_n$ independent of $a_0$ and $x_0$ and the sum
		ranges over all sequences $\mathbf{k}=(k_1,\dots, k_n)$ of
		nonnegative integers satisfying $k_1+k_2+\cdots+k_n=|a|-k-n.$ Since
		$\qbinom{a_0+a_i-1}{k_i}$ is a polynomial in $q^{a_0}$ of degree
		$k_i$, each summand in \eqref{e-midle} is a polynomial in $q^{a_0}$
		of degree at most $k_1+k_2+\cdots +k_n=|a|-k-n$, and so is the sum.
	\end{proof}
		
	\begin{lem}\label{lem-main1}
		If $T$ is a subset of $E_0$, and $a_1,\cdots a_n$ are nonnegative intergers, then
		\begin{align}\label{main-lem}
		\prod_{(i,j)\in T} \frac{x_{j}}{x_{i}}D_n(\mathbf{x},\mathbf{a},q) =\prod_{(i,j)\in E_0\overline{T}}(-1)^{|T|}\Big(\frac{x_i}{x_j}\Big)_{a_i}\Big(\frac{qx_j}{x_i}\Big)_{a_j-1},
		\end{align}
		where $E_0:=\{(i,j)\mid 0\leq i<j\leq n \}$.
	\end{lem}	
     \begin{proof}
     	 This was first proved in \cite[Proposition 2.1]{goul1985}.  For every element $(i,j)\in T$, it is clear that
     \begin{equation}
    \frac{x_j}{x_i}\Big(\frac{x_i}{x_j}\Big)_{a_i}\Big(\frac{qx_j}{x_i}\Big)_{a_j-1}=\frac{x_j}{x_i}\Big(1-\frac{x_i}{x_j}\Big)\Big(\frac{qx_i}{x_j}\Big)_{a_i-1}\Big(\frac{qx_j}{x_i}\Big)_{a_j-1}=(-1)\Big(\frac{qx_i}{x_j}\Big)_{a_i-1}\Big(\frac{x_j}{x_i}\Big)_{a_j}.
     \end{equation}
     \end{proof}
     \begin{lem}\label{zp}
     Let $0\leq b_1<b_2 <\cdots<b_s\leq a_1+a_2+\cdots+a_n$ and $b_1,b_2\cdots,b_s,a_1,a_2 \cdots a_n$ are nonnegative intergers. $F(a_1,a_2,\cdots,a_n)$ is a rational function independent of $q^{a_0}$,
     then $\mathbf{P_a(q^{a_0})}$ is a polynomial and the degree in $q^{a_0}$ is at most $|a|-s+1$,\\
     	where
     \begin{equation}\label{pa0}
     	\mathbf{P_a(q^{a_0})}=\frac{(1-q^{a_0})(q)_{a_0+a_1\cdots+a_n}}{(q)_{a_0}(q)_{a_1}\cdots(q)_{a_n}}\prod_{i=1}^{s}\frac{1}{1-q^{b_i+a_0}}F(a_1,a_2,\cdots,a_n).
     \end{equation}
     If $a_0\in \{0,-1,-2\cdots,-(a_1+a_2+a_3\cdots+a_n)\}\backslash\{-b_1,-b_2,\cdots,-b_m$\},  then $\mathbf{P_a(q^{a_0})=0}$.
     \end{lem}
     \begin{proof}
     	we can rewrite \eqref{pa0} as
     	\begin{align*}
     	    \mathbf{P_a(q^{a_0})}
     	=&\frac{(1-q^{a_0})(q)_{a_0+a_1\cdots+a_n}}{(q)_{a_0}(q)_{a_1}\cdots(q)_{a_n}}\prod_{i=1}^{s}\frac{1}{1-q^{a_0+b_i}}F(a_1,a_2,\cdots,a_n)\\
     	=&\frac{(1-q^{a_0+|a|})(1-q^{a_0+|a|-1})\cdots(1-q^{a_0}))}{(q)_{a_1}(q)_{a_2}\cdots(q)_{a_n}}\prod_{i=1}^{s}\frac{1}{1-q^{a_0+b_i}}F(a_1,a_2,\cdots,a_n)
     	\end{align*}
     	If  $\mathbf{a_0\in\{0,-1,-2\cdots,-(a_1+a_2+a_3\cdots+a_n)\}\backslash\{-b_1,-b_2,\cdots ,-b_s\}}$, then  $\mathbf{P_a(q^{a_0})=0}$.
     \end{proof}
	\section{Tournament and Laurent series}
		
	We let $K=\CC(q)$, and assume that 
	the field of iterated Laurent series $K\langle\!\langle x_n,
	x_{n-1},\ldots,x_0\rangle\!\rangle
	 =K(\!(x_n)\!)(\!(x_{n-1})\!)\cdots (\!(x_0)\!)$ includes all series. First we have a Laurent series in $x_0$ , then we have a Laurent series in $x_1$, and so on. 
	The work \cite{gessxin2006}  explained the motivation for choosing
	$K\langle\!\langle x_n, x_{n-1},\ldots,x_0\rangle\!\rangle$ as a
	working field. References \cite{xin2005} and \cite{xin2004} provide more
	detailed accounts of the properties of this field and its
	applications.
	
	The field of rational functions is a subfield
	of $K\langle\!\langle x_n, x_{n-1},\ldots,x_0\rangle\!\rangle$. Therefore, through a unique iterative Laurent series expansion, each rational function can be identified. If $i<j$ then
	$$\frac{1}{1-q^k x_i/x_j}=\sum_{l= 0}^\infty q^{kl} x_i^l x_j^{-l}.$$
	However, this expansion is not valid for $i>j $ and instead we
	use the expansion
	{\small$$ \frac{1}{1-q^k x_i/x_j}=\frac1{-q^k x_i/x_j(1-q^{-k}x_j/x_i)}
		=\sum_{l=0}^\infty -q^{-k(l+1)} x_i^{-l-1}x_j^{l+1}.$$}
	
	We shall use $\CT_{x_i} F(\mathbf{x})$ to denote the constant term of the series $F(\mathbf{x})$.
	It follows that
	\begin{equation}
	\label{e-ct} \CT_{x_i} \frac{1}{1-q^k x_i/x_j} =
	\begin{cases}
	1, & \text{ if }i<j, \\
	0, & \text{ if }i>j. \\
	\end{cases}
	\end{equation}
	The monomial $M=q^k x_i/x_j$ is called \emph{small} if $i<j$ and
	\emph{large} if $i>j$.  Therefore the constant term  of $1/(1-M)$ in $x_i$
	is $1$ if $M$ is small and $0$ if $M$ is large.
	
	 The constant term operators defined in this
	way have the commutativity property:
	$$\CT_{x_i} \CT _{x_j} F(\mathbf{x}) = \CT_{x_j} \CT_{x_i} F(\mathbf{x}).$$
	Commutativity implies that the constant term in a set of variables
	is well-defined, and this property will play a very important role in the proof of two main theorems. (Note that, by contrast, the constant term
	operators in \cite{zeil1999} do not commute.)

	The \emph{degree} of a rational function of $x$ is the degree in $x$
	of the numerator minus the degree  in $x$ of the denominator. A \emph{proper} (resp. \emph{almost proper}) rational function in $x$ has the property that its degree in
	$x$ is negative (resp. zero).
	
	Let
	\begin{align}\label{e-defF}
	F=\frac{p(x_k)}{x_k ^d \prod_{i=1}^m (1-x_k/\alpha_i)}
	\end{align}
	be a rational function of $x_k$, where $p(x_k)$ is a polynomial in
	$x_k$, and the $\alpha_i$ are distinct monomials, each of the form
	$x_t q^s$. Then the partial fraction decomposition of $F$ with
	respect to $x_k$ has the following form:
	{\small\begin{align}\label{e-defFs}
		F=p_0(x_k)+\frac{p_1(x_k)}{x_k^d}+\sum_{j=1}^m \frac{1}{1-
			x_k/\alpha_j}  \left. \left(\frac{p(x_k)}{x_k^d \prod_{i=1,i\ne j}^m
			(1-x_k/\alpha_i)}\right)\right|_{x_k=\alpha_j},
		\end{align}}
	where $p_0(x_k)$
	is a polynomial in $x_k$, and $p_1(x_k)$ is a polynomial in $x_k$ of
	degree less than $d$.
	
	We will use the following lemma as a
	basic tool for extracting constant terms. It is proven in \cite{zhou2009}.
	\begin{lem}\label{lem-almostprop}
		Let $F$ be as in \eqref{e-defF} and \eqref{e-defFs}.
		Then
		\begin{align}\label{e-almostprop}
		\CT_{x_k} F=p_0(0) +\sum_j  \bigl(F\,
		(1-x_k/\alpha_j)\bigr)\Bigr|_{x_k =\alpha_j},
		\end{align}
		where 
		the sum ranges over all $j$ such that $x_k/\alpha_j$ is small. In addition, if $F$ is proper in $x_k$, then $p_0(x_k)=0$; if $F$ is
		almost proper in $x_{k}$, then $x_k/\alpha_j$ must be large for all $j$ and therefore
		$p_0(x_k)=(-1)^m\prod_{i=1}^m\alpha_{i}\LC_{x_{k}}p(x_k)$, where
		$\LC_{x_k}$ means to take the leading coefficient with respect to
		$x_k$.
	\end{lem}
	
	The following slight {extension}  \cite[Lemma 4.2]{gessxin2006}
	plays an important role in our argument.
	\begin{lem}\label{lem-import1}
		Let $a_{1},\ldots,a_{s}$ be nonnegative integers. Then for arbitrary positive integers $k_{1},\ldots,k_{s}$ with $1\leq k_{i}\leq
		a_{1}+\cdots+a_{s}-1$ for all $i$, either $1\leq k_{i}\leq a_{i}-1$
		for some $i$ or $1-a_{j}\leq k_{i}-k_{j}\leq a_{i}-1$ for some $i<j$.
	\end{lem}
	\begin{proof}
		The basic idea is from \cite[Lemma 4.2]{gessxin2006}. But we can simplify it.
		Assume $k_1,\dots ,k_s$ satisfy that for all $i$, $a_i \le
		k_i\le a_1+\cdots +a_s-1$, and for all $i<j,$ either $k_i-k_j\ge
		a_i$ or $k_i-k_j\le -a_j$. Then we need to show that there exists an $i$ such that
		$k_{i}\geq a_{1}+\cdots+a_{s}$ .
        Let $k_1,k_2\cdots k_n$ be arranged from smallest to largest:
        $k_{i_1}\leq k_{i_2}\leq k_{i_3}\cdots \leq k_{i_s}$.
        It's easy to obtain $k_{i_m}+a_{i_{m+1}}\leq k_{i_{m+1}}$
        by checking two cases ($i_m>i_{m+1}$ and $i_m<i_{m+1}$). And $k_{i_1}\geq a_{i_1}$.
        This means
             \begin{align}
             	 k_{i_s}&\geq  k_{i_{1}}+a_{i_{2}}+\cdots+a_{i_{s}}\nonumber\\
             	&\geq a_{i_{1}}+a_{i_{2}}+a_{i_{3}}+\cdots+a_{i_{s}}\nonumber\\
             	&=a_{1}+a_{2}+\cdots+a_{s}. \nonumber
             \end{align}
             
        By assumption, $k_{i_s}\leq a_1+a_2+\cdots+a_s-1$, but $k_{i_s}\geq a_1+a_2+\cdots+a_s$.
        This completes the proof.
	\end{proof}

    \begin{lem}\label{non}
  Let $T$ be a tournament on $n$ vertices and  $R_T$ be the collection of nonempty subsets $R$ of the vertices of $T$ such that if $r \in R$ and $m \notin R$, then $(r,m) \in T$. If $T$ is non-transitive, then $|R_T| \leq n-2$.
    
    \end{lem}
    \begin{proof}
   We first note that if $R_1, R_2 \in R_T$ and $|R_1| \geq |R_2|$, then $R_1 \supseteq R_2$. This follows because if $r_2 \in R_2$, $r_2 \notin R_1$, then there is at least one $r_1 \in R_1$ that is not in $R_2$. Therefore we get the contradiction that $(r_2,r_1) \in T$ and $(r_1,r_2) \in T$. It follows that for each cardinality from 1 through $n$, there is at most one element of $R_T$ with that cardinality.
   
   Since $T$ is nontransitive, it contains a 3-cycle, $i \rightarrow j \rightarrow k \rightarrow i$. If $R \in R_T$ contains any one of these vertices, then it must contain all three. If not, then we can assume $i \in R$, $j \notin R$. If $k \in R$, then $k \rightarrow j$, while if $k \notin R$, then $i \rightarrow k$. Therefore, if $R_1$ contains $i$ while $R_2$ does not, then $|R_1| \geq |R_2| + 3$.
    \end{proof}
	
	\begin{lem}\label{sigma}
	Let T be a tournament and $\sigma$ is a permutation,
	then
	\begin{equation}
	\CT_{\mathbf{x}}\frac{x_m}{x_k}\prod_{(i,j)\in T}\Big(\frac{x_i}{x_j}\Big)_{a_i}\Big(\frac{qx_j}{x_i}\Big)_{a_j-1}=\CT_{\mathbf{x}}\frac{x_{\sigma(m)}}{x_{\sigma(k)}}\prod_{(i,j)\in T}\Big(\frac{x_{\sigma(i)}}{x_{\sigma(j)}}\Big)_{a_i}\Big(\frac{qx_{\sigma(j)}}{x_{\sigma(i)}}\Big)_{a_j-1}.
	\end{equation}
    \end{lem}
	\begin{proof}
	We replace $x_i$ with $x_{\sigma(i)}$. It is straightforward to check that the identity is true.
	\end{proof}	

     \section{ Nontransitive Condition} 
     	Define $Q_1(b)$ and $Q_2(b)$ to be
	\begin{align}
	Q_1(b):=\frac{x_0}{x_1}\prod_{(i,j)\in T_1}\frac{x_j}{x_i} \prod_{j=1}^{n}
	\left(\frac{x_0}{x_j}\right)_{\!\!\!-b}\left(\frac{x_j}{x_0}q\right)_{\!\!\!a_j-1}
	\prod_{1\leq i<j\leq n}
	\left(\frac{x_i}{x_j}\right)_{\!\!\!a_i}\left(\frac{x_j}{x_i}q\right)_{\!\!\!a_j-1},
	\end{align}
	\begin{align}
	Q_2(b):=\frac{x_0}{x_2}\prod_{(i,j)\in T_2}\frac{x_j}{x_i}  \prod_{j=1}^{n}
	\left(\frac{x_0}{x_j}\right)_{\!\!\!-b}\left(\frac{x_j}{x_0}q\right)_{\!\!\!a_j-1}
	\prod_{1\leq i<j\leq n}
	\left(\frac{x_i}{x_j}\right)_{\!\!\!a_i}\left(\frac{x_j}{x_i}q\right)_{\!\!\!a_j-1}.
	\end{align}
	If $b\ge 0$, then
	\begin{align}
	Q_1(b)=\frac{x_0}{x_1}\prod_{(i,j)\in T_1}\frac{x_j}{x_i} \prod_{j=1}^{n} \frac{(x_jq/x_0)_{a_j-1}}
	{\big(1-\frac{x_0}{x_jq}\big)\big(1-\frac{x_0}{x_jq^2}\big)\cdots
		\big(1-\frac{x_0}{x_jq^b}\big)}
	\prod_{1\leq i<j\leq n} \left(\frac{x_i}{x_j}\right)_{\!\!\!a_i}\left(\frac{x_j}{x_i}q\right)_{\!\!\!a_j-1},
	\end{align}
	\begin{align}
	Q_2(b)=\frac{x_0}{x_2}\prod_{(i,j)\in T_2}\frac{x_j}{x_i} \prod_{j=1}^{n} \frac{(x_jq/x_0)_{a_j-1}}
	{\big(1-\frac{x_0}{x_jq}\big)\big(1-\frac{x_0}{x_jq^2}\big)\cdots
		\big(1-\frac{x_0}{x_jq^b}\big)}
	\prod_{1\leq i<j\leq n} \left(\frac{x_i}{x_j}\right)_{\!\!\!a_i}\left(\frac{x_j}{x_i}q\right)_{\!\!\!a_j-1}.
	\end{align}
	
	Note that the degree in $x_0$ of $1-x_jq^i/x_0$ is zero, the degree in
$x_0$ of $Q_1(b)$ and $Q_2(b)$ is  $1-nb$. Thus $Q_1(b)$ and $Q_2(b)$ are proper when $b>0$ and
$n\ge 2$. Lemma \ref{lem-almostprop} gives
\begin{align}\label{qh2}
\CT_{x_0}Q_1(b)=\sum_{\substack{0<r_1\leq n,\\ 1\leq k_1\leq
b}}Q_1(b\mid r_1;k_1)
\end{align}
\begin{align}\label{qh3}
\CT_{x_0}Q_2(b)=\sum_{\substack{0<r_1\leq n,\\ 1\leq k_1\leq
b}}Q_2(b\mid r_1;k_1),
\end{align}
where
$$
Q_1(b\mid r_1;k_1)=Q_1(b)\left(1-\frac{x_0}{x_{r_1}q^{k_1}}\right)\bigg|_{x_0=x_{r_1}q^{k_1}}$$
$$
Q_2(b\mid r_1;k_1)=Q_2(b)\left(1-\frac{x_0}{x_{r_1}q^{k_1}}\right)\bigg|_{x_0=x_{r_1}q^{k_1}}.
$$
For each term in \eqref{qh2} and \eqref{qh3}, we will extract the constant term in $x_{r_1}$, then perform further constant term extractions, evaluating one variable at each step. We introduce some notation from \cite{gessxin2006} to keep track of the terms we select.
	
Let $F$ be a rational function of $x_0, x_1, \ldots , x_n$. Given a subset $R\subseteq \{1,2,\cdots,n\}$, let $\mathbf{r}$
be the sequence of the elements of $R$ in increasing order, $\mathbf{r} = (r_1, r_2, \ldots , r_s)$,
and let k be a sequence of positive integer constants, $\mathbf{k} = (k_1, k_2, \ldots, k_s)$. Let
$E_{\mathbf{r,k}} F$ be the result of replacing $x_{r_i}$ with $x_{r_s}q^{k_s-k_i}$ in $F$ for $i = 0, 1, \ldots, s -1,$ where we set $r_0= k_0 = 0$. Then for
$0 < r_1 < r_2 < \cdots < r_s \leq n$ and $0 < k_i \leq b$, we
define
	\begin{align}\label{q1rk}
	Q_1(b\mid\mathbf{r;k})=Q_1(b\mid r_1,\ldots,r_s;k_1,\ldots,k_s)=
	E_{\mathbf{r,k}}\left[Q_1(b)\prod_{i=1}^{s}\Big(1-\frac{x_0}{x_{r_i}q^{k_i}}\Big)\right],
	\end{align}
	\begin{align}\label{q2rk}
	Q_2(b\mid\mathbf{r;k})=Q_2(b\mid r_1,\ldots,r_s;k_1,\ldots,k_s)=
	E_{\mathbf{r,k}}\left[Q_2(b)\prod_{i=1}^{s}\Big(1-\frac{x_0}{x_{r_i}q^{k_i}}\Big)\right].
	\end{align}
	Note that the product on the right-hand side of \eqref{q1rk} and \eqref{q2rk} cancels all the factors in the
	denominator of $Q_1$ and $Q_2$ that would be taken to zero by $\Erk$.
	\begin{lem}\label{lem-import2}
    
    Given a tournament $E\overline{T}$ where $T$ is a subset of $E$, then $\Erk\frac{A}{B}$ has degree $0$ in $x_{r_s}$
    if and only if for any $r\in R,m\notin R,(r,m)\in E\overline{T}$. Otherwise $\Erk\frac{A}{B}$ is of   negative degree in $x_{r_s}$,
    where 
    $$A:=\prod_{(i,j)\in T}\frac{x_j}{x_i},\quad  B:=\prod_{1\leq i<j\leq n}1-\frac{q^{a_j}x_j}{x_i}.$$ 
	\end{lem}
	\begin{proof}
		First we note that $\Erk(1-\frac{q^{a_j}x_j}{x_i})$ is of degree $1$ in
		$x_{r_s}$  if $j \in R$ and $i\notin R$, otherwise the degree is $0$.
		Thus the term in $\Erk B$ that can contribute to the degree of $x_{r_s}$ can be written as $$\prod_{k=1}^{s}\prod_{\mbox{$\tiny \begin{array}{c}
					1\leq i\leq r_k\\i\notin R \end{array} $}}\Erk\left[(1-\frac{q^{a_{r_k}}x_{r_k}}{x_i})\right].$$
		The total degree of $x_{r_s}$ in $$\prod_{\tiny	\begin{array}{c}			
				1\leq i<r_{k}\\
				i\notin R \end{array}}
		\Erk \left[(1-\frac{q^{a_{r_k}}x_{r_k}}{x_i})\right]$$
		is $r_k-k$. Therefore the degree of $x_{r_s}$ in $\Erk B$ is
		\begin{equation}\label{deg4}
		    (r_1+r_2\ldots+r_s)-\frac{(s+1)s}{2}.
		\end{equation}
		Next, we consider the degree of $x_{r_s}$ in $A$. When $(i,j)\in T$, the degree of $x_{r_s}$ in $\Erk\frac{x_j}{x_i}$ is 1 when $j\in R,i\notin R$, $-1$ when $i\in R$ and $j\notin R$, and $0$ otherwise. Therefore, the part of $A$ that can contribute to the degree of $x_{r_s}$ is
		$$\prod_{i=1}^{s}\prod_{\mbox{\tiny$\begin{array}{c}
				l\notin R\\
				(l,r_i)\in T
				\end{array}$}}\Erk\left[\frac{x_{r_i}}{x_l}\right]\prod_{\mbox{\tiny$\begin{array}{c}
				l\notin R\\
				(r_i,l)\in T
				\end{array}$}}\Erk\left[\frac{x_{l}}{x_{r_i}}\right].$$
 		It is easy to see that the degree of $x_{r_s}$ in
		   $$ \prod_{(l,r_i)\in T}\Erk
		   \left[\frac{x_{r_i}}{x_l}\right]\prod_{(r_i,l)\in T}\Erk
		   \left[\frac{x_{l}}{x_{r_i}}\right] $$
		 is at most $r_i-i$ and equals $r_i-i$ if and only if there is no $l \notin R$ for which $(r_i,l)\in T$. It follows that the degree of $x_{r_s}$ in $\Erk A$ is at most $r_1+r_2+\cdots +r_s- (s+1)s/2$ with equality if and only if $r \in R$ and $m \notin R$, implies that $(m,r)\in T$. Since edges in $T$ are reversed in $E \overline{T}$, we must have $(r,m)\in E\overline{T}.$ 
	\end{proof}
    
    In preparation for Lemma~\ref{lem-lead3}, we recall that $T_1$ is a subset of $\{(i, j) \mid 2 \le i <
    j \le n\}$ and $T_2$ is a subset of $\{(i, j) \mid 3 \le i < j \le n\}$. We define the following
    collections of subsets of the set of vertices $1$ through $n$:
    \begin{itemize}
    	   \item $R_{1}$ is the collection of subsets $R$ that satisfy the condition that if $r \in R$, and $m \notin R$, then $(m, r) \in T_{1}$ together with the set $\{l_0\}$ where $l_0$ beats all vertices except for $1$ and is larger than $1$.
    	\item $R_{2}$ is the collection of subsets $R$ that satisfy the condition that if $r \in R$, and $m \notin R$, then $(m, r) \in T_{2}$ together with the set $\{1,l_0\}$ where $l_0 \ge 3$ and $l_0$ beats all vertices except for $1$ and $2$.
    \end{itemize}

	\begin{lem}\label{lem-lead3}
		Let $R'=R\cup \{0\}=\{r_0,r_1,r_2\cdots,r_n\}$. Then the rational functions
		$Q_1(b\mid \mathbf{r;k})$ and $Q_2(b\mid \mathbf{r;k})$ have the following properties:
		\begin{itemize}
			\item[{\bf i}] If $1\leq k_i\leq a_{r_1}+\cdots+a_{r_s}-1$ for all $i$ with $1\leq i\leq s$,
			then $ Q_1(b\mid \mathbf{r;k})=0$ and $ Q_2(b\mid \mathbf{r;k})=0$.
			\item[{\bf ii}] If $k_i\geq a_{r_1}+\cdots+a_{r_s}$ for some $i$ with $1
			\leq i \leq s<n$, and if
			\begin{align}
			b\neq \sum_{r\in R} a_r,
			\end{align}
			when $R \in R_1$, then
			\begin{align}\label{qh5}
			\CT_{x_{r_s}}Q_1(b\mid \mathbf{r;k})= \sum_{\substack{r_s<r_{s+1}\leq n,\\
					1\leq k_{s+1}\leq b}}
			Q_1(b\mid r_1,\ldots,r_s,r_{s+1};k_1,\ldots,k_s,k_{s+1}).
			\end{align}
				\item[{\bf iii}] If $k_i\geq a_{r_1}+\cdots+a_{r_s}$ for some $i$ with $1
			\leq i \leq s<n$, and if
			\begin{align}
			b\neq \sum_{r\in R}a_r,
			\end{align}
			when $R \in R_2$, then
			\begin{align}\label{qh6}
			\CT_{x_{r_s}}Q_2(b\mid \mathbf{r;k})= \sum_{\substack{r_s<r_{s+1}\leq n,\\
					1\leq k_{s+1}\leq b}}
			Q_2(b\mid r_1,\ldots,r_s,r_{s+1};k_1,\ldots,k_s,k_{s+1}).
			\end{align}
		\end{itemize}		
	\end{lem}
	\begin{proof}
		[Proof of property {\bf(i)}] By Lemma \ref{lem-import1}, either $1\le
		k_i\le a_{r_i}-1$ for some $i$ with $1\le i \le s$, or $1-a_{r_j}\le
		k_i-k_j\le a_{r_i}-1$ for some $i<j$, because the exceptional case can not happen. If $1\le k_i\le a_{r_i}-1$ then
		$\Qrki$ and $\Qrkii$ have the factor
		$$\Erk
		\left[\lrq{x_{r_i}}{x_{0}}{q}_{\!\!a_{r_i}-1} \right]
		=\lrq{x_{r_s}q^{k_s-k_i}}{x_{r_s}q^{k_s}}{q}_{\!\!a_{r_i}-1}
		=(q^{1-k_i})_{a_{r_i}-1}=0.$$
		
		If $1-a_{r_j}\le k_i-k_j\le a_{r_i}-1$ where $i<j$ then $\Qrki$ and $\Qrkii$ have
		the factor
		$$\Erk\, \left[\lrq{x_{r_i}}{x_{r_j}}{}_{\!\!a_{r_i}}\!\!\lrq{x_{r_j}}{x_{r_i}}{q}_{\!\!a_{r_j-1}}\right]=\left(q^{k_j-k_i}\right)_{a_{r_i}} \left(q^{k_i-k_j+1}\right)_{a_{r_j}-1}. $$
		If $k_j \le k_i,$ then the first factor is $0$. If $k_j> k_i,$ then the second factor is $0$.\\
		\smallskip
		\noindent\emph{Proof of property {\bf(ii)} and {\bf(iii).}} Note that since $b\geq k_i$
		for all $i$, the hypothesis implies that   $b\geq a_{r_1}+\cdots
		+a_{r_s}$.
		
		We claim that $Q_1(b\mid \mathbf{r;k})$  and $Q_2(b\mid \mathbf{r;k})$ are proper in $x_{r_s}$. To do this, we rewrite $\Qrki$ as $(M_1NC_1)/(A_1DB_1)$ and $\Qrkii$ as $(M_2NC_2)/(A_2DB_2)$ , in which $N,$ $D,$ $M_1,$ $M_2,$ $A_1,$ $A_2,$ $B_1,$ $B_2,$ $C_1,$ $C_2$ are defined by
		
		$$ N=\Erk\, \left[\prod_{j=1}^n\lrq{x_j}{x_0}q_{\!\!a_j-1}
		\cdot \prod_{\substack{1\le i, j\le n\\ j\neq i}}
		\left(\frac{x_i}{x_j}\,q^{\chi(i>j)}\right)_{\!\!a_i}\right],$$
     	$$D=\Erk\, \left[\prod_{j=1}^n \lrq{x_0}{x_jq^b}{}_{\!\!\!b}\biggm/\prod_{i=1}^s\left(1-\frac
		{x_0}{x_{r_i}q^{k_i}}\right)\right],$$
		$$M_1=\Erk \left[\prod_{(i,j)\in T_1}\frac{x_j}{x_i}\right],$$
		$$M_2=\Erk \left[\prod_{(i,j)\in T_2}\frac{x_j}{x_i} \right],$$
	    $$A_1=\Erk\,\left[\prod_{2\leq i<j\leq n}(1-\frac{q^{a_j}x_j}{x_i})\right],$$
	    $$A_2=\Erk\,\left[\prod_{3\leq i<j\leq n}(1-\frac{q^{a_j}x_j}{x_i})\right],$$
		$$B_1=\Erk \left[\prod_{2\leq j\leq n}(1-\frac{q^{a_j}x_j}{x_1}) \right],$$
		$$B_2=\Erk \left[\prod_{2\leq j\leq n}(1-\frac{q^{a_j}x_j}{x_1})\prod_{3\leq j\leq n}(1-\frac{q^{a_j}x_j}{x_2}) \right],$$
	    $$C_1=\Erk \left[\frac{x_0}{x_1} \right],$$
		$$C_2=\Erk \left[\frac{x_0}{x_2} \right],$$
		where $\chi(S)$ is $1$ if the statement $S$ is true, and $0$ otherwise. Note that $R=\{r_1,r_2,\cdots,r_s\}$ and 
		$R'=R\cup\{r_0\}= \{r_0,r_1,\dots,r_s\}$ where $r_0=0$. Then the degree of $x_{r_s}$ in
		$$\Erk\, \left[\left(1-\frac{x_i}{x_j}q^m\right)\right]$$
		is $1$ if $i\in R'$ and $j\not\in R'$, and is $0$ otherwise, as is
		easily seen by checking the four cases. Clearly the degree of $x_{r_s}$
		in $\Erk\, x_i^{b_i}$ is $b_i$ if $i\in R'$ and is $0$ otherwise.
		Thus the part of $N$ contributing to the degree in
		$x_{r_s}$ is
		
		$$E_{\mathbf{r},\mathbf{k}}\left[\prod_{i=1}^s \prod_{j\ne r_0,\dots ,r_s}
		\left(\frac{x_{r_i}}{x_j}q^{\chi(r_i>j)}\right)_{\!\!a_{r_i}}\right],$$
		which has degree
		$
		(n-s)(a_{r_1}+\cdots +a_{r_s}).
		$
		The part of $D$ contributing to the degree of $x_{r_s}$ are
		$$E_{\mathbf{r},\mathbf{k}}\left[\prod_{j\ne r_0,\dots, r_s}\lrq{x_0}{x_jq^b}{}_{\!\!b}\right],$$
		which has degree $(n-s)b$.
		The degree of $x_{r_s}$ in $B_1$ is
		 $$   \Erk \left[\prod_{1\leq i\leq s \atop 1<r_i}(1-\frac{q^{a_{r_i}}x_{r_i}}{x_1})\right].$$
		Thus the part  $B_1$ of  degree in $x_{r_s}$ is $\chi(1\notin R)\cdot s$.
		The part of $B_2$ contributing to the degree of $x_{r_s}$ is
		$$
		    \Erk \left[\prod_{1\leq i\leq s \atop 1<r_i}(1-\frac{q^{a_{r_i}}x_{r_i}}{x_1})\prod_{1\leq i\leq s \atop 2<r_i}(1-\frac{q^{a_{r_i}}x_{r_i}}{x_2})\right].
		$$
		So  the degree of $x_{r_s}$ in $B_2$ is $\chi(1\notin R)\cdot s+\chi(2\notin R)\cdot (s-\chi(1\in R))$.
		The part of $C_1$ contributing to the degree of $x_{r_s}$ is
		$1-\chi(1\in R)$ and the degree of $x_{r_s}$ in $C_2$  is $1-\chi(2 \in R)$.
		Thus the total degree of $x_{r_s}$ in $\Qrki$  is
		\begin{align}
		d_1=(n-s)(a_{r_1}+\cdots +a_{r_s}-b)-\chi(1\notin R)\cdot s+1-\chi(1\in R)+d_{t_1}.
		\end{align}
		The total degree of $\Qrkii$ in $x_{r_s}$ is
		   \begin{equation}
		   \begin{split}
		     d_{2}=&(n-s)(a_{r_1}+\cdots +a_{r_s}-b)-\chi(1\notin R)\cdot s-\\&\chi(2\notin R)\cdot (s-\chi(1\in R))+1-\chi(2\in R)+d_{t_2}.  
		   \end{split}
		   \end{equation}
		where $d_{t_1}$ is the degree of $x_{r_s}$ in $\frac{M_1}{A_1}$  and $d_{t_2}$ is the degree of $x_{r_s}$ in $\frac{M_2}{A_2}$.
		
		By Lemma \ref{lem-import2}, we note that $d_{t_1}$ and $d_{t_2}$ are at most $0$.
		
		We first observe that $-\chi(1 \notin R) · s + 1- \chi(1 \in R) \le 0$ whether or not $1 \in R$.
		Next, we note that $-\chi(1 \notin R) · s + 1- \chi(1 \in R) = 0$ if and only if $1 \in R$ or $s=1$. In these cases, by Lemma~\ref{lem-import2}, $d_{t_1}$
		is negative unless $R \in R_1$, in which case it is $0$. But
		if $R \in R_1$, then $b > a_{r_1} + \cdots a_{r_s}$
		so that $(n - s)(a_{r_1} + \cdots a_{r_s} -b)$ is negative.
		It follows that $d_1$ is strictly negative.

		We check that $-\chi(1 \notin R) · s -\chi(2 \notin R) (s - \chi(1 \in R)) + 1 -\chi(2 \in R) \le 0 $ in
		each of the four cases in which $1$ and $2$ are or are not elements of $R$. Next, $-\chi(1 \notin R) · s -\chi(2 \notin R) (s - \chi(1 \in R)) + 1 -\chi(2 \in R) = 0 $ if and only if
		$1\in R$ and $2\in R$ or $1\in R$ and $s\le 2$. In these cases, again by Lemma \ref{lem-import2}, $d_{t_2}$ is negative unless $R \in R_2$, in which
		case it is $0$. But if $R \in  R_2$, then $b > a_{r_1} +\cdots +a_{r_s}$, so that $(n-s)(a_{r_1} +\cdots + a_{r_s}-b)$ is negative. It follows that $d_2$ is strictly negative.
		
		So $\Qrki$ and $\Qrkii$ are
		proper in $x_{r_s}$. Next, we apply Lemma \ref{lem-almostprop}. For
		any rational function $F$ of $x_{r_s}$ and integers $j$ and $k$, let
		$T_{j,k} F$ be the result of replacing $x_{r_s}$ with
		$x_{j}q^{k-k_s}$ in $F$. Since $x_{r_s}q^{k_s}/(x_jq^k)$ is large when $j<r_s$ and is small
		when $j>r_s$, Lemma~\ref{lem-almostprop}
		gives
		\begin{equation}
		\label{e-TQ3} \CT_{x_{r_s}} \Qrki =\sum_{r_s < r_{s+1}\le n\atop 1\le
			k_{{s+1}}\le b} T_{r_{s+1},k_{s+1}} \left[\Qrki
		\left(1-\frac{x_{r_s}q^{k_s}}{x_{r_{s+1}}q^{k_{s+1}}}\right)\right],\\
		\end{equation}
		\begin{equation}\label{e_TQ4}
	     \CT_{x_{r_s}} \Qrkii =\sum_{r_s < r_{s+1}\le n\atop 1\le
			k_{{s+1}}\le b} T_{r_{s+1},k_{s+1}} \left[\Qrkii
		\left(1-\frac{x_{r_s}q^{k_s}}{x_{r_{s+1}}q^{k_{s+1}}}\right)\right].\\
		\end{equation}
		It is necessary to show that the right-hand side of $\eqref{e-TQ3}$ and $\eqref{e_TQ4}$ is equal to
		the right-hand side of \eqref{qh5} and \eqref{qh6}. Set $\br'=(r_1,\dots,
		r_s, r_{s+1})$ and $\bk'=(k_1,\dots, k_s, k_{s+1})$. Then the
		equality follows easily from the identity
		\begin{equation}
		\label{e-TE} T_{r_{s+1},k_{s+1}}\circ \Erk= E_{\br',\bk'}.
		\end{equation}
		To check that \eqref{e-TE} holds, we see
		$$
		(T_{r_{s+1},k_{s+1}}\circ \Erk)\, x_{r_i}
		=T_{r_{s+1},k_{s+1}}\, \left[ x_{r_s}q^{k_s-k_i}\right]
		= x_{r_{s+1}}q^{k_{s+1}-k_i}= E_{\br',\bk'}\,  x_{r_i},
		$$
		and if $j\notin\{r_0,\dots, r_s\}$ then $(T_{r_{s+1},k_{s+1}}\circ
		\Erk)\, x_{j}=x_j=  E_{\br',\bk'}\, x_{j}$.
        \end{proof}
        
        \begin{thm} \label{eq1}
        If $E\overline{T_1}$ is a nontransitive tournament, then
        $$\CT_{\mathbf{x}} \frac{x_0}{x_1} \prod_{(i,j)\in T_1} \frac{x_j}{x_i}D_n(\mathbf{x},\mathbf{a},q)=0.$$
        \end{thm}
        \begin{proof}
        Applying Lemma~\ref{pd}, LHS is a polynomial of degree at most $|a|-n-1$ in $q^{a_0}$.
       We will proceed by induction on $n-s$ to show that $\CT_{x}\Qrki=0$ when $b\in \{1,2,\cdots, |a|\}$ and
        $b\neq \sum_{r\in R}a_r$, where $R \in R_1$ as defined before the statement of Lemma~\ref{lem-lead3}.
        
        If $n-s=0$, then Lemma~\ref{lem-lead3} {\bf(i)} implies $\Qrki=0$. We next show that if the theorem is true for $n-s=k$, it is also true for $n-s=k+1$.
       If $b<a_{r_1}+a_{r_2}\cdots+a_{r_s}$, 
       Lemma~\ref{lem-lead3}{\bf(i)} gives $\Qrki=0$, otherwise we apply
       Lemma~\ref{lem-lead3}{\bf(ii)},
       \begin{align}
			\CT_{x_{r_s}}Q_1(b\mid \mathbf{r;k})= \sum_{\substack{r_s<r_{s+1}\leq n,\\
					1\leq k_{s+1}\leq b}}
			Q_1(b\mid r_1,\ldots,r_s,r_{s+1};k_1,\ldots,k_s,k_{s+1}).
			\end{align}
			By Lemma~\ref{zp}, $CT_xQ_1(b)$ as a polynomial in $q^{-b}$ has $|a|-|R_1|$ zeroes. To find
			an upper bound on $|R_1|$, we observe that it is the union of two sets, $N_1=\{R\mid \text{if  $r\in R$ and $m\notin R$}, \text{then }  (r,m)\in E\overline{T_1} \}$ and $N_2 = \{R \mid |R|=1$ and the single element of $R$ is larger than $1$ and beats all vertices except $1\}$. By Lemma~\ref{non},  $|N_1| \le n-2$. Since there is at most one vertex that beats all vertices except for the vertex $1$ and is larger than $1$, $|N_2| \le 1$, and therefore
				$|R_1| \le n - 1$ and $ CT_x Q_1(b)$ has at least $|a| - n + 1$ zeroes, implying that it is
				identically $0$.
        \end{proof}
        
         \begin{thm}\label{eq2}
         If $E\overline{T_2}$ is a nontransitive tournament, then
        $$\CT_{\mathbf{x}} \frac{x_0}{x_2} \prod_{(i,j)\in T_2} \frac{x_j}{x_i}D_n(\mathbf{x},\mathbf{a},q)=0.$$
        \end{thm}
        \begin{proof}
        Applying Lemma~\ref{pd}, LHS is a polynomial of degree at most $|a|-n-1$ in $q^{a_0}$.
        We will proceed by induction on $n-s$ to show that $\CT_{x}\Qrki=0$ when $b\in \{1,2,\cdots, |a|\}$ and
        $b\neq \sum_{r\in R}a_r$, where $R \in R_2$ as defined before the statement of Lemma~\ref{lem-lead3}.
        
        If $n-s=0$, then Lemma~\ref{lem-lead3} {\bf(i)} implies $\Qrkii=0$. We next show that if the theorem is true for $n-s=k$, it is also true for $n-s=k+1$.
        If $b<a_{r_1}+a_{r_2}\cdots+a_{r_s}$, 
        applying
        Lemma~\ref{lem-lead3}{\bf(i)}, we obtain $\Qrkii=0$, otherwise we apply
       Lemma~\ref{lem-lead3}{\bf(iii)},
       \begin{align}
			\CT_{x_{r_s}}Q_2(b\mid \mathbf{r;k})= \sum_{\substack{r_s<r_{s+1}\leq n,\\
					1\leq k_{s+1}\leq b}}
			Q_2(b\mid r_1,\ldots,r_s,r_{s+1};k_1,\ldots,k_s,k_{s+1}).
			\end{align}
			
		 	By Lemma~\ref{zp}, $CT_xQ_1(b)$ as a polynomial in $q^{-b}$ has $|a|-|R_2|$ zeroes. To find
		 an upper bound on $|R_2|$, we observe that it is the union of two sets,
		  $N_3=\{R\mid \text{if $r\in R$, $m\notin R$, then}$ $(r,m)\in E\overline{T_2}\}$ and
		 $N_4=\{R\mid R=\{1,l_0\}$ where $l_0\ge 3$ and $l_0$ beats all vertices except for $1$ and $2\}$. By Lemma~\ref{non},  $|N_3| \le n-2$. Since there is at most one vertex that beats all vertices except for the vertices $1$ and $2$ and is larger than $2$, $|N_4| \le 1$, and therefore
		 $|R_2| \le n - 1$ and $ CT_x Q_2(b)$ has at least $|a| - n + 1$ zeroes, implying that it is
		 identically $0$.
        \end{proof}

 \section{Transitive Condition}
    In this section, we prove the transitive condition in these Theorems \ref{mainthm2} and \ref{mainthm3} and Corollary~\ref{main3}.
    We will first prove these theorems when $T_1$ and $T_2$ are empty sets and will use
    Lemma~\ref{lem-main1} to prove the general transitive case as well as Corollary~\ref{main3}.
	Let $b$ be an integer. Define $Q_3(b)$ and $Q_4(b)$ to be
	\begin{align}
	Q_3(b):=\frac{x_0}{x_1} \prod_{j=1}^{n}
	\left(\frac{x_0}{x_j}\right)_{\!\!\!-b}\left(\frac{x_j}{x_0}q\right)_{\!\!\!a_j-1}
	\prod_{1\leq i<j\leq n}
	\left(\frac{x_i}{x_j}\right)_{\!\!\!a_i}\left(\frac{x_j}{x_i}q\right)_{\!\!\!a_j-1},
	\end{align}
	\begin{align}
	Q_4(b):=\frac{x_0}{x_2} \prod_{j=1}^{n}
	\left(\frac{x_0}{x_j}\right)_{\!\!\!-b}\left(\frac{x_j}{x_0}q\right)_{\!\!\!a_j-1}
	\prod_{1\leq i<j\leq n}
	\left(\frac{x_i}{x_j}\right)_{\!\!\!a_i}\left(\frac{x_j}{x_i}q\right)_{\!\!\!a_j-1}.
	\end{align}
	If $b\ge 0$, then
	\begin{align}
	Q_3(b)=\frac{x_0}{x_1}\prod_{j=1}^{n} \frac{(x_jq/x_0)_{a_j-1}}
	{\big(1-\frac{x_0}{x_jq}\big)\big(1-\frac{x_0}{x_jq^2}\big)\cdots
		\big(1-\frac{x_0}{x_jq^b}\big)}
	\prod_{1\leq i<j\leq n} \left(\frac{x_i}{x_j}\right)_{\!\!\!a_i}\left(\frac{x_j}{x_i}q\right)_{\!\!\!a_j-1},
	\end{align}
	\begin{align}
	Q_4(b)=\frac{x_0}{x_2}\prod_{j=1}^{n} \frac{(x_jq/x_0)_{a_j-1}}
	{\big(1-\frac{x_0}{x_jq}\big)\big(1-\frac{x_0}{x_jq^2}\big)\cdots
		\big(1-\frac{x_0}{x_jq^b}\big)}
	\prod_{1\leq i<j\leq n} \left(\frac{x_i}{x_j}\right)_{\!\!\!a_i}\left(\frac{x_j}{x_i}q\right)_{\!\!\!a_j-1}.
	\end{align}
	\begin{align}
	Q_3(b\mid\mathbf{r;k})=Q_3(b\mid r_1,\ldots,r_s;k_1,\ldots,k_s)=
	E_{\mathbf{r,k}}\left[Q_3(b)\prod_{i=1}^{s}\Big(1-\frac{x_0}{x_{r_i}q^{k_i}}\Big)\right],
	\end{align}
	\begin{align}
	Q_4(b\mid\mathbf{r;k})=Q_4(b\mid r_1,\ldots,r_s;k_1,\ldots,k_s)=
	E_{\mathbf{r,k}}\left[Q_4(b)\prod_{i=1}^{s}\Big(1-\frac{x_0}{x_{r_i}q^{k_i}}\Big)\right].
	\end{align}

	Let $s_k=\sum_{i=1}^k a_i$. For $Q_3$ we consider values of $b$ in $\{1,2\cdots s_n\}$ / $\{s_1,\cdots s_n,a_2\}$. For $Q_4(b)$, we consider values of $b$ in $\{1,2\cdots s_n\}$ / $\{s_1,\cdots s_n,a_1+a_3\}$. We also consider $b=0$ for both $Q_3(b)$ and $Q_4(b)$.
		\begin{lem}\label{x1}
			Let $a_1,a_2,\cdots,a_n$ be nonnegative integers and set $a_0=0$. We have the following identity
		\begin{equation}
		  \CT_{\mathbf{x}}Q_3(0)=\frac{(q^{a_1}-q^{a_2})}{1-q^{a_2}}\frac{(q)_{a_1+a_2+\cdots+a_n}}{(q)_{a_1}(q)_{a_2}+\cdots+(q)_{a_n}}\prod_{i=1}^n\frac{1-q^{a_i}}{1-q^{s_i}}.
		\end{equation}
		\end{lem}
        \begin{proof}
        Applying Theorem~\ref{t-main-thm}:
        \begin{align*}
            \CT_{\mathbf{x}}Q_3(0) =& \CT_{\mathbf{x}}\frac{x_0}{x_1}\prod_{i=1}^n(\frac{qx_i}{x_0})_{a_i-1}\prod_{1\leq i<j \leq n}(\frac{x_i}{x_j})_{a_i}(\frac{qx_j}{x_i})_{a_j-1} \\
                                 =&\CT_{\mathbf{x}}\sum_{i=1}^n \frac{q^{a_i}-q}{1-q}\frac{x_i}{x_1}\prod_{1\leq i<j \leq n}(\frac{x_i}{x_j})_{a_i}(\frac{qx_j}{x_i})_{a_j-1} \\
                                 =&\frac{q^{a_1}-q}{1-q}\frac{(q)_{a_1+a_2+\cdots+a_n}}{(q)_{a_1}(q)_{a_2}+\cdots+(q)_{a_n}}\prod_{i=1}^n\frac{1-q^{a_i}}{1-q^{s_i}}\\
                                 &- \frac{q^{a_2}-q}{1-q}\frac{(q)_{a_1+a_2+\cdots+a_n}}{(q)_{a_1}(q)_{a_2}+\cdots+(q)_{a_n}}\prod_{i=1}^n\frac{1-q^{a_i}}{1-q^{s_i'}}\\
                                 =&\frac{(q^{a_1}-q^{a_2})}{1-q^{a_2}}\frac{(q)_{a_1+a_2+\cdots+a_n}}{(q)_{a_1}(q)_{a_2}+\cdots+(q)_{a_n}}\prod_{i=1}^n\frac{1-q^{a_i}}{1-q^{s_i}},
        \end{align*}
        where $s_1'=a_2$, $s_k'=\sum_{i=1}^k a_i(k\geq 2).$
        \end{proof}
        \begin{lem}\label{x12}
        Let $a_0=-b$, then
        	\begin{equation}
	    \CT_{\mathbf{x}}Q_3(b)=
     \Big(\frac{q^{a_1}-q^{a_2}}{1-q^{a_0+a_2}}\Big)\frac{(q)_{a_0+a_1+\cdots+a_n}}{(q)_{a_0}(q)_{a_1}\cdots (q)_{a_n}}\prod_{i=1}^n\frac{1-q^{a_i}}{(1-q^{a_0+s_i})}.
	\end{equation}
        \end{lem}
        \begin{proof}
        By Lemma~\ref{pd}, LHS is a polynomial in $q^{-b}$ of degree at most $|a|-n-1$.
        From Lemma~\ref{zp}, RHS is a polynomial in $q^{a_0}$ with zeroes at $q^i$ for $ i\in \{-1,-2,\cdots,-s_n\}/\{-s_1,-s_2,\cdots,\\-s_n,-a_2\}.$
        
        We show that $\CT_{x}Q_3(b)$ is zero when $b\in \{1,2,\cdots,s_n\}$/$\{s_1,s_2\cdots,s_n,a_2\}$. We proceed by induction on $n-s$ to show that $\CT_{x}\Qrka=0$. If $n-s=0$, then Lemma~\ref{lem-lead3} {\bf{(i)}} gives $\Qrka=0$. We now assume that $\Qrka=0$ when $n-s=k$. To verify that it is also true when $n-s=k+1$, we note that if $b<a_{r_1}+a_{r_2}+\cdots + a_{r_{s+1}}$, then Lemma~\ref{lem-lead3} {\bf{(ii)}} gives $\Qrka=0$. Otherwise we apply Lemma~\ref{lem-lead3} {\bf{(ii)}} to obtain
         \begin{align}
        \CT_{x_{r_s}}Q_3(b\mid \mathbf{r;k})= \sum_{\substack{r_s<r_{s+1}\leq n,\\
        		1\leq k_{s+1}\leq b}}
        Q_3(b\mid r_1,\ldots,r_s,r_{s+1};k_1,\ldots,k_s,k_{s+1}).
        \end{align}  
        Now, every term $\CT_{x}\Qrka=0$. Since both sides are polynomials of degree at most $|a|-n-1$ zeroes and, by Lemma~\ref{x1}, they agree at $q^{a_0}=q^0=1$, they must be equal.
        \end{proof}
        \begin{lem}\label{x2}
        	Let $a_1,a_2,\cdots,a_n$ be nonnegative integers, then we have the next identity,
        \begin{equation}
            \CT_{\mathbf{x}}Q_4(0)=\frac{(1+q^{a_1})(q^{a_2}-q^{a_3})}{1-q^{a_1+a_3}}\frac{(q)_{a_1+a_2+\cdots+a_n}}{(q)_{a_1}(q)_{a_2}+\cdots+(q)_{a_n}}\prod_{i=1}^n\frac{1-q^{a_i}}{1-q^{s_i}}.
        \end{equation}
        \end{lem}
        \begin{proof}
       By Theorem~\ref{t-main-thm} and Lemma~\ref{x12}:
        \begin{align*}
            \CT_{\mathbf{x}}Q_4(0) =& \CT_{\mathbf{x}}\frac{x_0}{x_2}\prod_{i=1}^n\Big(\frac{qx_i}{x_0}\Big)_{a_i-1}\prod_{1\leq i<j \leq n}\Big(\frac{x_i}{x_j}\Big)_{a_i}\Big(\frac{qx_j}{x_i}\Big)_{a_j-1} \\
                                 =&\CT_{\mathbf{x}}\sum_{i=1}^n \frac{q^{a_i}-q}{1-q}\frac{x_i}{x_2}\prod_{1\leq i<j \leq n}\Big(\frac{x_i}{x_j}\Big)_{a_i}\Big(\frac{qx_j}{x_i}\Big)_{a_j-1} \\
                                  =&\frac{(q^{a_1}-q)(q^{a_2}-q^{a_3})}{(1-q)(1-q^{a_1+a_3})}\frac{(q)_{a_1+a_2+\cdots+a_n}}{(q)_{a_1}(q)_{a_2}+\cdots+(q)_{a_n}}\prod_{i=1}^n\frac{1-q^{a_i}}{1-q^{s_i}}\\
                                  &+\frac{q^{a_2}-q}{1-q}\frac{(q)_{a_1+a_2+\cdots+a_n}}{(q)_{a_1}(q)_{a_2}+\cdots+(q)_{a_n}}\prod_{i=1}^n\frac{1-q^{a_i}}{1-q^{s_i}}\\
                                  &- \frac{q^{a_3}-q}{1-q}\frac{(q)_{a_1+a_2+\cdots+a_n}}{(q)_{a_1}(q)_{a_2}+\cdots+(q)_{a_n}}\prod_{i=1}^n\frac{1-q^{a_i}}{1-q^{s_i'}}\\
                                  &=\frac{(1+q^{a_1})(q^{a_2}-q^{a_3})}{1-q^{a_1+a_3}}\frac{(q)_{a_1+a_2+\cdots+a_n}}{(q)_{a_1}(q)_{a_2}+\cdots+(q)_{a_n}}\prod_{i=1}^n\frac{1-q^{a_i}}{1-q^{s_i}}, \\
        \end{align*}
        where $s_2'=a_1+a_3$ and $s_k'=\sum_{i=1}^k a_i(k\neq 2)$.
        \end{proof}
        \begin{lem}\label{x02}
        Let $a_0=-b$, then 
        \begin{equation}
	    \CT_{\mathbf{x}}Q_4(b)=
	\frac{(1+q^{a_1})(q^{a_2}-q^{a_3})}{(1-q^{a_0+a_1+a_3})}
        \frac{(q)_{a_0+a_1\cdots +a_n}}{(q)_{a_0}(q)_{a_1}\cdots+(q)_{a_n}}
        \prod_{i=1}^{n}\frac{1-q^{a_i}}{1-q^{a_0+s_i}}.
	\end{equation}
        \end{lem}
        \begin{proof} By Lemma~\ref{pd}, LHS is a polynomial in $q^{-b}$ of degree at most $|a|-n-1$.
        	From Lemma~\ref{zp}, RHS is a polynomial in $q^{a_0}$ with zeroes at $q^i$ for $ i\in \{-1,-2,\cdots,-s_n\}/\{-s_1,-s_2,\cdots,\\-s_n,-a_1-a_3\}.$
        	
        	We show that $\CT_{x}Q_4(b)$ is zero when $b\in \{1,2,\cdots,s_n\}$/$\{s_1,s_2\cdots,s_n,a_1+a_3\}$. We proceed by induction on $n-s$ to show that $\CT_{x}\Qrka=0$. If $n-s=0$, then Lemma~\ref{lem-lead3}{\bf{(i)}} gives $\Qrka=0$. We now assume that $\Qrka=0$ when $n-s=k$. To verify that it is also true when $n-s=k+1$, we note that if $b<a_{r_1}+a_{r_2}+\cdots + a_{r_{s+1}}$, then Lemma~\ref{lem-lead3}{\bf{(i)}} gives $\Qrka=0$. Otherwise we apply Lemma~\ref{lem-lead3}{\bf{(iii)}} to obtain
        	\begin{align}
        	\CT_{x_{r_s}}Q_4(b\mid \mathbf{r;k})= \sum_{\substack{r_s<r_{s+1}\leq n,\\
        			1\leq k_{s+1}\leq b}}
        	Q_4(b\mid r_1,\ldots,r_s,r_{s+1};k_1,\ldots,k_s,k_{s+1}).
        	\end{align}	 	
        	Now, every term $\CT_{x}\Qrkb=0$. Since both sides are polynomials of degree at most $|a|-n-1$ zeroes and, by Lemma~\ref{x2}, they agree at $q^{a_0}=q^0=1$, they must be equal.
      \end{proof}
  
      {\bf{Conclusion to the proofs of Theorems \ref{mainthm2} and \ref{mainthm3} and Corollary~\ref{main3}.} }
      We note that corollary~\ref{main3} extends Lemmas \ref{x02} and \ref{x12} to any transitive tournaments since we can let $\sigma$ be any winner permutation in $E\overline{T_1}$ and $E\overline{T_2}$. So we only need to prove
      corollary~\ref{main3}.\\
      (i) Applying Lemma~\ref{lem-main1}, Lemma~\ref{sigma} and Lemma~\ref{x12}, then
        \begin{align*}
             & \CT_{\mathbf{x}}\frac{x_{\sigma(1)}}{x_{\sigma(2)}}\prod_{(i,j)\in T}\frac{x_i}{x_j}
	\prod_{1\leq i<j\leq n}\Big(\frac{x_i}{x_j}\Big)_{a_i}\Big(\frac{qx_j}{x_i}\Big)_{a_j-1}\\
	 =& (-1)^{|T|}\CT_{\mathbf{x}}\frac{x_{\sigma(1)}}{x_{\sigma(2)}}\prod_{1\leq i<j\leq n}\Big(\frac{x_{\sigma(i)}}{x_{\sigma(j)}}\Big)_{a_\sigma(i)}\Big(\frac{qx_{\sigma(j)}}{x_{\sigma(i)}}\Big)_{a_{\sigma(j)-1}}\\
	  =& (-1)^{|T|}\CT_{\mathbf{x}}\frac{x_1}{x_2}\prod_{1\leq i<j\leq n}\Big(\frac{x_i}{x_j}\Big)_{a_{\sigma(i)}}\Big(\frac{qx_j}{x_i}\Big)_{a_{\sigma(j)-1}}\\
	 =& (-1)^{|T|}\frac{q^{a_{\sigma(2)}}-q^{a_{\sigma(3)}}}{1-q^{a_{\sigma(1)}+a_{\sigma(3)}}}\frac{(q)_{a_1+a_2+\cdots+a_n}}{(q)_{a_1}(q)_{a_2}\cdots (q)_{a_n}} \prod_{i=1}^n\frac{1-q^{a_{\sigma(i)}}}{1-q^{\Psi_i}}.\\
        \end{align*}
        (ii) We apply Lemma~\ref{lem-main1}, Lemma~\ref{sigma} and Lemma~\ref{x02}, then
           \begin{align*}
             & \CT_{\mathbf{x}}\frac{x_{\sigma(1)}}{x_{\sigma(3)}}\prod_{(i,j)\in T}\frac{x_i}{x_j}
	\prod_{1\leq i<j\leq n}\Big(\frac{x_i}{x_j}\Big)_{a_i}\Big(\frac{qx_j}{x_i}\Big)_{a_j-1}\\
	 =& (-1)^{|T|}\CT_{\mathbf{x}}\frac{x_{\sigma(1)}}{x_{\sigma(3)}}\prod_{1\leq i<j\leq n}\Big(\frac{x_{\sigma(i)}}{x_{\sigma(j)}}\Big)_{a_{\sigma(i)}}\Big(\frac{qx_{\sigma(j)}}{x_{\sigma(i)}}\Big)_{a_{\sigma(j)-1}}\\
	 =& (-1)^{|T|}\CT_{\mathbf{x}}\frac{x_1}{x_3}\prod_{1\leq i<j\leq n}\Big(\frac{x_i}{x_j}\Big)_{a_{\sigma(i)}}\Big(\frac{qx_j}{x_i}\Big)_{a_{\sigma(j)-1}}\\
	 =& (-1)^{|T|}
	\frac{(1+q^{a_{\sigma(2)}})(q^{a_{\sigma(3)}}-q^{a_{\sigma(4)}})}{1-q^{a_{\sigma(1)}+a_{\sigma(2)}+a_{\sigma(4)}}}\frac{(q)_{a_1+a_2+\cdots+a_n}}{(q)_{a_1}(q)_{a_2}\cdots (q)_{a_n}} \prod_{i=1}^n\frac{1-q^{a_{\sigma(i)}}}{1-q^{\Psi_i}}. \\
	\end{align*}
	
	\vspace{.2cm} \noindent{\bf Acknowledgments.} I would like to express my very great appreciation to  Professor Mourad E. H. Ismail  for his valuable and constructive suggestions during the planning and development of this research work. I am pleased to acknowledge Professor David Bressoud for his wise and valuable advice and comments. I sincerely thank Professor Ira M. Gessel for his useful suggestions and comments. I also would like to acknowledge the helpful guidance from Dr.Yue Zhou. Last but not least, I am very grateful to a very kind referee who reads the paper carefully and provides many valuable suggestions which give a substantial improvement for this paper.

	\end{document}